\documentclass{birkjour}
\usepackage[all]{xy}
\usepackage{cleveref}

\def\O{\mathcal O}

\def\im{\operatorname{Im}}
\def\id{\operatorname{id}}

\def\AA{{\mathbb A}}

\def\FF{{\mathbb F}}

\def\tt{{\mathfrak T}}

\newcommand{\Cc}{\mathcal{C}}

\def\O{{\bf O}}

\def\*C{{}^*\hspace*{-1pt}{\Cc}}

\def\text#1{{\rm {\rm #1}}}

 \def\1{\mathbf{1}}

  \def\la{\triangleright}

\def\id{\mathrm{id}}

\def\sl#1#2{\mathrm{sl}{(#1 , #2)}}
\newcommand{\np}{\raisebox{0.2ex}{\scalebox{0.6}{$\,\bullet\,$}}}

\newtheorem{theorem}{Theorem}[section]

\newtheorem{lemma}[theorem]{Lemma}
\newtheorem{proposition}[theorem]{Proposition}
\theoremstyle{definition}
\newtheorem{definition}[theorem]{Definition}
\theoremstyle{remark}

\newtheorem*{example}{Example}
\numberwithin{equation}{section}

\newcounter{rlist}

\begin{document}

\title{A note on Rota-Baxter affgebras.}

 \author{Tomasz Brzezi\'nski}
 \address{Department of Mathematics, Swansea University\\
Fabian Way,
Swansea
SA1 8EN, U.K.  \\
\and\\ Department of Mathematics, University of Bia\l ystok\\ K.\ Cio\l kowskiego 1M, 15--245 Bia\l ystok, Poland} 
  \email{T.Brzezinski@swansea.ac.uk}   

\thanks{The research is partially supported by the National Science Centre, Poland, through the WEAVE-UNISONO grant no.\ 2023/05/Y/ST1/00046.}


\begin{abstract}
An affinization of a concept of a Rota-Baxter operator or Rota-Baxter affgebras are discussed. Rota-Baxter Lie affgebras and their relation to Rota-Baxter Lie algebras are treated in more detail. In particular it is shown that every Rota-Baxter Lie affgebra is isomorphic to a Rota-Baxter Lie affgebra with a Lie algebra fibre over which the affine Rota-Baxter operator is linear. 
\end{abstract}

\maketitle

\section{Introduction}
This note is a contribution to the algebraic \textit{affinization programme} \cite{BrzRadRam:Lei}. The purpose of this programme is to formulate familiar and important also for applications in physics linear algebraic structures such as an associative, Lie, Leibniz, Poisson etc.\ algebras in an affine way, i.e.\ without making choice of the zero vector (or a frame of reference in physics applications). A general idea is to replace vector spaces by affine spaces and (multi-)linear transformations by (multi-)affine transformations. This involves working with the \textit{absolute} definition of an affine space, i.e.\ a definition which makes no use of the underlying vector space, while still allowing for recovery of the vector space once a choice of a point is made the \textit{linear} or \textit{tangent fibre}. Next, algebraic conditions for typically bi-affine operations, such as the associative law or the Jacobi identity, are redrafted in purely affine setup while making sure that the corresponding bilinear operations on fibres satisfy the familiar linear conditions.   

The algebraic affinization programme reviewed above has its origins in the Tulczyjew programme of a frame independent approach to classical mechanics \cite{Tul:fra}, with the formulation of Lie brackets on affine spaces (albeit not defined in the absolute way) in \cite{GraGra:Lie}, where the term \textit{affgebra} was also coined to describe an algebra structure on an affine space, as its first exponent. On algebraic side the programme has been boosted by finding in \cite{Brz:tru}, \cite{Brz:par} that an affine framework connects rings with \textit{braces} \cite{Rum:bra}, \cite{CedJesOkn:bra}. Consequently, affine versions of associative rings and algebras  \cite{AndBrzRyb:ext}, Lie algebras  \cite{BrzPap:Lie},\cite{AndBrzRad:Lie}, Leibniz algebras  \cite{BrzRadRam:Lei}, Poisson algebras \cite{BrzRadRam:Poi}, and most pertinently for this note, associative Rota-Baxter algebras \cite{Cht:Rot} have been described recently. In this paper we introduce not necessarily associative Rota-Baxter affgebras, relate them to Rota-Baxter algebras and have a closer look at the Lie version.

\section{Affine spaces and Rota-Baxter algebras}\label{sec.RB}
In absolute terms,  an {affine space} over a field $\FF$ is a set $\mathfrak{A}$ together with two ternary operations, $\langle -,-,-\rangle: \mathfrak{A}^3\to \mathfrak{A}$ and $-\la_--: \FF\times \mathfrak{A}^2\to \mathfrak{A}$, such that,  for all $a,b,c,d,e\in \mathfrak{A}$, $\alpha,\beta\in \FF$,
$$
\begin{aligned}
    &\langle\langle a,b,c\rangle,d,e\rangle = \langle a,b,\langle c,d,e\rangle\rangle, \quad \langle a,a,b\rangle = b, \quad \langle a,b,c\rangle = \langle c,b,a\rangle.\\
&(\alpha+\beta)\triangleright_a b = \langle \alpha \triangleright_a b,  a, \beta \triangleright_a b\rangle, \;\;\; \alpha \triangleright_a \langle b,c,d\rangle = \langle \alpha \triangleright_a b,  \alpha \triangleright_a c, \alpha \triangleright_a d\rangle,\\
&(\alpha \beta)\triangleright_a b = \alpha \triangleright_a (\beta\triangleright_a b), \quad 1\triangleright_a b = b, \quad
\alpha \triangleright_a b = \langle \alpha \triangleright_c b, \alpha \triangleright_c a, a\rangle. 
\end{aligned}
$$
We note in passing that the first four conditions imply that $0 \triangleright_a b =a$. An affine map from $\mathfrak{A}$ to $\mathfrak{B}$ is a function $f:\mathfrak{A}\to \mathfrak{B}$ that preserves both ternary operations, i.e.,
\begin{equation}\label{mor.act}
f\left(\langle a,b,c\rangle\right) = \langle f(a), f(b), f(c), \rangle, \quad f(\alpha \triangleright_a b) = \alpha \triangleright_{f(a)}f(b) .
\end{equation}
It is important to observe that, for all $o\in \mathfrak{A}$, $\mathfrak{A}$  with operations 
\begin{equation}\label{tangent}
a+b := \langle a,o,b\rangle, \qquad \alpha a = \alpha\la_o a,
\end{equation}
is a vector space, which is denoted by $T_o\mathfrak{A}$ and called a \textit{linear} or \textit{tangent fibre} at $o$. In terms of binary operations \cref{tangent}, the ternary operations come out as
\begin{equation}\label{triples}
    \langle a,b,c\rangle = a-b+c, \qquad \alpha\la_a b = \alpha b + (1-\alpha) a.
\end{equation}
Consequently, the additive formulae \cref{triples} are \textbf{independent} of the choice of $o$, despite the fact that the binary operations \cref{tangent} depend on it. The formulae 
\cref{triples} can be extended to any barycentric combinations $\sum_i \alpha_i a_i$, with $\sum_i\alpha_i =1$. Being aware that binary operations are more strongly imprinted in our minds than ternary ones, we will use the combinations of binary operations \cref{triples}
to denote the ternary ones from now on (with one exception in the next paragraph). We would like to stress that this is for convenience of  the reader not out of necessity. On the other hand, any vector space can be understood as an affine space by operations \cref{triples}.

For all $o,e\in \mathfrak{A}$, the vector spaces $T_o\mathfrak{A}$ and $T_e\mathfrak{A}$ are isomorphic canonically by the  \textit{translation map}
\begin{equation}
    \tau_o^e: T_o\mathfrak{A}\to T_e\mathfrak{A}, \qquad a\mapsto \langle a,o,e\rangle.
\end{equation}
Affine versions of  algebraic structures are required to be isomorphic on fibres as linear structures through this isomorphism (and they are). Finally, we note that any map $f$ satisfying conditions \cref{mor.act} induces a linear map $T_of: T_o\mathfrak{A}\to T_e\mathfrak{B}$, $a\mapsto f(a)-f(o)$, which is referred to as the \textit{linearization} of $f$ at $o$.

Rota-Baxter operators and algebras were introduced in \cite{Bax:ana}, \cite{Rot:Bax} and they have found widespread and remarkable applications in analysis, probability, Hamiltonian integrable systems, operad theory etc.\ (see e.g. \cite{Guo:Rot}). By a \textit{Rota-Baxter algebra} of a scalar weight $\sigma$ we mean a vector space $A$ with a bilinear operation $\np$ and a linear operator $B$ such that, for all $a,b\in A$,
\begin{equation}\label{RB}
    B(a)\np B(b) = B(B(a)\np b - \sigma\, a\np b +a\np B(b)).
\end{equation}
There are two separate classes of such algebras. Either $\sigma =0$ or if $\sigma \neq 0$ $B$ can be rescaled so that $\sigma =1$. In the latter case, the right hand side of \cref{RB} takes the form of an affine combination as in \cref{triples}, and this is the case to which we will restrict ourselves in the present note and skip stressing the weight. 

Depending on the nature of $\np$, one can deal with associative Rota-Baxter algebras, Rota-Baxter Lie algebras etc. For any $B$ satisfying \cref{RB}, the operator $R=2B-\sigma \id$ satisfies the relation
\begin{equation}\label{mCYB}
    R(a)\np R(b) = R(R(a)\np b) - \sigma^2\, a\np b +R(a\np R(b)).
\end{equation}
In the case of the Rota-Baxter Lie algebra \cref{mCYB} is a modified classical Yang-Baxter equation as discussed in \cite{Sem:cla}. As before, by rescaling this splits into the case $\sigma =0$ (the classical Yang-Baxter equation) and  $\sigma =1$ (the modified classical Yang-Baxter equation). In the latter instance we note that $R$ is defined by an affine combination of $B$ and the identity map, and hence it lends itself keenly to the affinization.
\section{Rota-Baxter affgebras vis-\'a-vis Rota-Baxter algebras}
\begin{definition}\label{def.RB.aff}
    By a {\em Rota-Baxter affgebra} we mean an affine space $\mathfrak{A}$ together with a bi-affine operation $\np:\mathfrak{A}^2\to \mathfrak{A}$ and an affine operator $\mathcal{B}:\mathfrak{A}\to \mathfrak{A}$ such that, for all $a,b,\in \mathfrak{A}$,
    \begin{equation}\label{RB-1}
    \mathcal{B}(a)\np \mathcal{B}(b) = \mathcal{B}(\mathcal{B}(a)\np b -  a\np b +a\np \mathcal{B}(b)).
\end{equation}
\end{definition}
\begin{theorem}\label{thm.RB}
    For a Rota-Baxter affgebra $(\mathfrak{A},\np,\mathcal{B})$ and any $o\in \mathfrak{A}$, let
    $$
    \cdot := T_o\np: T_o\mathfrak{A}\times T_o\mathfrak{A}\to T_o\mathfrak{A}, \quad a\cdot b = a\np b -a\np o +o\np o - o\np b,
    $$
    $$
    B := T_o\mathcal{B} : T_o\mathfrak{A}\to T_o\mathfrak{A}, \quad B(a) = \mathcal{B}(a) - \mathcal{B}(o),
    $$
be linearizations of $\np$ and $\mathcal{B}$ to $T_o\mathfrak{A}$. Then $(T_o\mathfrak{A},\cdot,B)$ is a Rota-Baxter algebra.

Conversely, given an algebra $(A,\cdot)$ and an operator $B$ on $A$, view $A$ as an affine space via \cref{triples} and define affinizations of $\cdot$ and $B$ by
$$
a\np b = a\cdot b +\lambda(b) + \mu(a) +s, \quad \mathcal{B}(a) = B(a) + r,
$$
where $\lambda, \mu$  are linear operators on $A$ and $r,s\in A$. Then $(A,\np,\mathcal{B})$ is a Rota-Baxter affgebra if and only if $(A,\cdot, B)$ is a Rota-Baxter algebra (of weight 1), and, for all $a\in A$,
\begin{subequations}\label{RB.cond}
    \begin{equation}\label{const}
     r -r\cdot r = \bar B (\lambda(r) + s +\mu(r)),
    \end{equation}
    \begin{equation}\label{right}
         \bar B \lambda B(a) =B(r\cdot a ) - r\cdot B(a),
    \end{equation}
    \begin{equation}\label{left}
         \bar B \mu B(a) =B(a\cdot r) - B(a)\cdot r,
    \end{equation}
\end{subequations}
where $\bar B = \id-B$.
\end{theorem}
\begin{proof}
    The first statement is proven by a straightforward calculation with key steps as follows ($\cdot$ denoted by juxtaposition): 
    $$
    \begin{aligned}
      B(a)B(b) = &~   \mathcal{B}(a)\np \mathcal{B}(b) - \mathcal{B}(a)\np \mathcal{B}(o) -\mathcal{B}(o)\np \mathcal{B}(b) + \mathcal{B}(o)\np \mathcal{B}(o)\\
      = &~ \mathcal{B}(\mathcal{B}(a)\np b -  a\np b +a\np \mathcal{B}(b)) - \mathcal{B}(\mathcal{B}(a)\np o -  a\np o +a\np \mathcal{B}(o))\\
      &-\mathcal{B}(\mathcal{B}(o)\np b -  o\np b +o\np \mathcal{B}(b)) +\mathcal{B}(\mathcal{B}(o)\np o -  o\np o +o\np \mathcal{B}(o)) - \mathcal{B}(o) \\
      = &~ \mathcal{B}(\mathcal{B}(a) b - \mathcal{B}(o) b -  ab +a \mathcal{B}(b) - o \mathcal{B}(b)) -\mathcal{B}(o)\\
      = &~ B(B(a)b - ab+ aB(b)).
    \end{aligned}
    $$
    The point of note here is the appearance of the term $-\mathcal{B}(o)$ as a result of the fact that $\mathcal{B}$ respects only affine or barycentric combinations of elements of $\mathfrak{A}$.

    To prove the second statement, one evaluates  \cref{RB-1} at $a=b=0$ to obtain \cref{const}, then taking the latter into account, the evaluation of \cref{RB-1} at $a=0$ produces \cref{right}, then at $b=0$ gives \cref{left} and one is left with the Rota-Baxter condition for $B$.
\end{proof}

    In fact Theorem~\labelcref{thm.RB} establishes a two-way correspondence between Rota-Baxter affgebras and Rota-Baxter algebras with additional data $\lambda,\mu,s,r$ on the vector fibres. One can easily check that $\mathfrak{A}$ can be recovered from $A:=T_o\mathfrak{A}$ by setting 
    $$
    \lambda(a) = o\np a -o\np o, \quad \mu(a) = a\np o -o\np o, \quad s = o\np o, \quad r = \mathcal{B}(o).
    $$
This allows one to relate completely Rota-Baxter affgebras to Rota-Baxter algebras and thus to produce examples or even classifications of the former based on examples of the latter. The maps $\lambda, \mu$ satisfy additional conditions depending on the type of the algebra. We will turn our attention to this presently, but first we discuss a couple of examples that are independent of the algebra type. 

\begin{example}
    If $B=\id$, then conditions \cref{right}  and \cref{left} are automatically satisfied (as $\bar B =0$), while \cref{const} means that $r$ is an idempotent in $(A,\cdot)$. Hence, for all $o\in \mathfrak{A}$, $\mathcal{B} = \id -o+ r$ with $r$ an idempotent in $T_o\mathfrak{A}$ makes $(\mathfrak{A},\np)$ a Rota-Baxter affgebra. 
    
    At the other extreme, $(T_o\mathfrak{A}, B=o)$ is a Rota-Baxter algebra. Again,  \cref{right}  and \cref{left} are automatically fulfilled. 
The condition \labelcref{const} means that $r$ is an idempotent in $(\mathfrak{A},\np)$. Thus, for any idempotent $r$ in  $(\mathfrak{A},\np)$,  $(\mathfrak{A},\np, \mathcal{B}=r)$ is a Rota-Baxter affgebra.
\end{example}

As explained in \cite{AndBrzRyb:ext}, an affgebra $(\mathfrak{A},\np)$ associated to an algebra $(A,\cdot)$ and data $\lambda,\mu,s$ is associative if and only if  for all $a,b\in A$,
\begin{equation}\label{homothety}
\begin{aligned}
    \lambda(ab)=&\lambda(a)b, \quad \mu(ab) = a\mu(b), \quad a\lambda(b)=\mu(a)b, \quad \lambda\mu=\mu\lambda,\\ 
    \lambda(s)&= \mu(s),\quad  \lambda^2(a) = \lambda(a) + sa, \quad  \mu^2(a) = \mu(a) + as, 
\end{aligned}
\end{equation}
where the multiplication $\cdot$ is denoted by juxtaposition. Therefore, by \cref{thm.RB} \textit{associative Rota-Baxter affgebras}  $(\mathfrak{A},\np,\mathcal{B})$ are in one-to-one correspondence with associative Rota-Baxter algebras $(A,\cdot, B)$ equipped with data $\lambda,\mu,r,s$ that satisfy conditions \cref{RB.cond} and \cref{homothety}.

\begin{example}[Associative Rota-Baxter affgebras with trivial algebra fibres] 
There are four non-isomorphic associative affgebras corresponding to an algebra $A$ with the zero multiplication, $a\cdot b =0$, for all $a,b\in A$; see \cite{AndBrzRyb:ext}. In all cases $s=0$. Any linear operator $B$ on $A$ is a Rota-Baxter operator. The restrictions on corresponding affine Rota-Baxter operators are listed in the table below.
\begin{table}[h]
        \centering
        \begin{tabular}{|c|c|c|}
        \hline
        Affine multiplication & Data $\lambda,\mu$ & Rota-Baxter restrictions\\
       \hline \hline
           $ a\np b =a$ & $\lambda=0$, $\mu=\id$ &  $B(r)=0$, \quad $B^2=B$ \\
           \hline 
           $ a\np b =b$ & $\lambda=\id$, $\mu=0$ &  $B(r)=0$, \quad $B^2=B$ \\ \hline 
           $ a\np b =a+b$ & $\lambda=\mu=\id$ &  $2B(r)=r$, \quad $B^2=B$ \\ \hline 
           $ a\np b =0$ & $\lambda=\mu=0$ &  $r=0$ \\ \hline 
        \end{tabular}
    \end{table}

\end{example}

It is natural to attribute the term homomorphism to any affine transformation between Rota-Baxter algebras that preserves the binary operations and Rota-Baxter operators. In case $\mathfrak{A}$ is determined by a Rota-Baxter algebra $(A,\cdot,B)$ and the datum $(\lambda,\mu,s,r)$ in Theorem~\labelcref{thm.RB}, while $\mathfrak{A'}$ is determined by a Rota-Baxter algebra $(A',\cdot,B')$ and the datum $(\lambda',\mu',s',r')$, homomorphisms $\varphi: \mathfrak{A}\to \mathfrak{A'}$ correspond bijectively to pairs consisting of a linear map $\psi: A\to A'$ and $q'\in A'$ such that, for all $a,b\in A$,
$$
\psi(a\cdot b) = \psi(a)\cdot\psi(b), \qquad B'\psi = \psi B,
$$
i.e.\ $\psi$ is a homomorphism of Rota-Baxter algebras, and
\begin{subequations}\label{RB-hom}
    \begin{equation}\label{hom.affgebra}
\begin{aligned}
    &\psi(s) = (\lambda'+\mu')(q')+ {q'}^2-q' +s',\\
    \psi\mu(a) = \psi(a)&q' +\mu'\psi(a), \quad \psi\lambda(a) = q'\psi(a) +\lambda'\psi(a),
\end{aligned}     
    \end{equation}
    \begin{equation}\label{hom.RB}
        \psi(r) = r'-\bar B (q').
    \end{equation}
\end{subequations}
\Cref{hom.affgebra} guarantees that $\varphi$ is a homomorphism of affgebras (i.e.\ the binary operations are preserved by $\varphi$), and \cref{hom.RB} ensures the preservation of affine Rota-Baxter operators. The necessity and sufficiency of these conditions are established by a routine evaluation of the required equalities at $a=b=0$, $a=0$ and $b=0$ as in the proof of the second statement in Theorem~\labelcref{thm.RB}.

\section{Rota-Baxter Lie affgebras with simple fibres}
In this section we describe the relation between Rota-Baxter Lie affgebras and Rota-Baxter Lie algebras and illustrate it in detail by describing fully Rota-Baxter Lie affgebras with simple Lie fibres.
\begin{definition}[\cite{BrzPap:Lie}, \cite{AndBrzRad:Lie}]\label{def.Lie}
    A \textit{Lie affgebra} is an affine space $\mathfrak{a}$ together with a bi-affine operation $\{-,-\}$, called the \textit{Lie bracket}, such that, for all $X,Y,Z \in \mathfrak{a}$, 
    \begin{subequations}\label{Lie}
        \begin{equation}\label{anti}
            \{X,Y\} =  \{X,X\} - \{Y,X\} + \{Y,Y\},
        \end{equation}
        \begin{equation}\label{Jacobi}
        \begin{aligned}
            \{X,\{Y,Z\}\} &- \{X,\{X,X\}\} + \{Y,\{Z,X\}\}\\ 
            &= \{Y,\{Y,Y\}\} -\{Z,\{X,Y\}\} + \{Z,\{Z,Z\}\}.
        \end{aligned}
            \end{equation}
    \end{subequations}
\end{definition}
The first condition is the affine antisymmetry and the second one is an affine version of the Jacobi identity.  The connection between Lie affgebras and Lie algebras is studied in \cite{AndBrzRad:Lie}. In particular, the linearization of the affine Lie bracket $\{-,-\}$ on each fibre induces a Lie algebra; these Lie algebras are mutually isomorphic in a canonical way. On the other hand, if $(\mathfrak{g},[-,-])$, then the affinization of the Lie bracket $[-,-]$:
$$
\{X,Y\} = [X,Y] + \kappa(X) + \lambda(Y-X) +s,
$$
where $s\in\mathfrak{g}$ and $\kappa,\lambda \in \mathrm{Lin}(\mathfrak{g})$, becomes an affine Lie bracket on $\mathfrak{g}$ if and only if
\begin{equation}\label{gen.der}
  \lambda\left([X,Y]\right) = [\lambda(X),Y] - [X,\kappa(Y)] + [X,\lambda(Y)],
\end{equation}
for all $X,Y\in \mathfrak{g}$. Every Lie affgebra with the Lie algebra fibre $\mathfrak{g}$ is isomorphic (as a Lie affgebra) to $\mathfrak{g}$ with the bracket as in \labelcref{gen.der}; it is then denoted by $\mathfrak{a}(\mathfrak{g};\kappa,\lambda ,s)$.

\begin{definition}\label{def.RBLie}
    A Rota-Baxter affgebra $(\mathfrak{a},\{-,-\},\mathcal{B})$ where $\{-,-\}$ is an affine Lie bracket  is called a \textit{Rota-Baxter Lie affgebra}.
\end{definition}

In view of the discussion around \cref{mCYB} one can easily  justify the following affinization of the modified classical Yang-Baxter equation.

\begin{lemma}\label{lem.mCYB}
     If $\mathrm{char}\, \mathbb{F} \neq 2$, then there is a one-to-one correspondence between Rota-Baxter Lie affgebras $(\mathfrak{a},\{-,-\},\mathcal{B})$ and solutions $\mathcal{R}$ to the \emph{affine modified classical Yang-Baxter equation}
     \begin{equation}\label{a.mCYBE}
          \{\mathcal{R}(X), \mathcal{R}(Y)\} = \mathcal{R}(\{\mathcal{R}(X), Y\}) -  \{X, Y\} +\mathcal{R}(\{X, \mathcal{R}(Y)\}).
     \end{equation}
\end{lemma}
\begin{proof}
    The correspondence is given by 
    $$
    \begin{aligned}
        \mathcal{B}\mapsto \mathcal{R} = \mathcal{B}- \id + \mathcal{B}, \qquad \mathcal{R}(X)\mapsto \mathcal{B}(X) &= \frac12 \la_X \mathcal{R}(X)\\
        &= \frac12 \mathcal{R}(X) - \frac12 X + X.
    \end{aligned}  
    $$
    This presentation is independent of the choice of $o$ in $\mathfrak{a}$.
\end{proof}

We can now apply Theorem~\labelcref{thm.RB} to construct Rota-Baxter Lie affgebras. Given a Rota-Baxter operator $B$ on a Lie algebra $\mathfrak{g}$, i.e.\ a linear operator $B:\mathfrak{g}\to \mathfrak{g}$, such that, for all $X,Y\in\mathfrak{g}$,
\begin{equation}\label{RB.op.Lie}
    [B(X),B(Y)] = B\left([B(X),Y]-[X,Y]+ [X,B(Y)]\right),
\end{equation}
as well as data $\kappa,\lambda, s$ leading to the Lie affgebra $\mathfrak{a}(\mathfrak{g};\kappa,\lambda,s)$, the affine operator $\mathcal{B}$, $X\mapsto B(X) + r$ is an affine Rota-Baxter operator on $\mathfrak{a}(\mathfrak{g};\kappa,\lambda,s)$, provided that
\begin{equation}\label{RB.Lie.con}
    r=\bar B(\kappa(r) +s) , \quad \bar B\kappa B = 0, \quad  \bar B\lambda B =B\mathrm{ad}_r - \mathrm{ad}_rB,
\end{equation}
where $\mathrm{ad}_r=[r,-]$ is the adjoint action of $r$ on $\mathfrak{g}$ and $\bar B = \id-B$.

\begin{lemma}\label{lem.Lie}
    The conditions \labelcref{RB.Lie.con} are equivalent to 
    \begin{equation}\label{RB.Lie.con.alt}
    r=\bar B\left(\kappa^n(r) +\sum_{l=0}^{n-1}\kappa^l(s)\right) , \quad \bar B\kappa B = 0, \quad  \bar B(\lambda +\mathrm{ad}_{\kappa(r)+s})B =0,
\end{equation}
for all positive integers $n$.
\end{lemma}
\begin{proof}
    The alternative form of the condition for $\lambda$ is obtained by combining the third with the first conditions in \cref{RB.Lie.con}, and the Rota-Baxter equation \labelcref{RB.op.Lie}. The rule for $r^n$ follows recursively from the first equation in  \labelcref{RB.Lie.con} and can be proven by induction.
\end{proof}

With a Rota-Baxter operator $B$ on a Lie algebra $\mathfrak{g}$ one can associate four Lie subalgebras of $\mathfrak{g}$:
$$
\mathfrak{g}_+ = \im B, \quad \mathfrak{g}_- = \im \bar B, \quad \mathfrak{h}_+ = \ker \bar B, \quad \mathfrak{h}_- = \ker B;
$$
see \cite[Section~4]{Sem:cla}. Obviously, $\mathfrak{g} = \mathfrak{g}_+ +  \mathfrak{g}_- $ and $\mathfrak{h}_\pm\subseteq \mathfrak{g}_\pm$. Having this notation at hand we can phrase the second and third of \cref{RB.Lie.con.alt} as $\kappa(\mathfrak{g}_+)\subseteq \mathfrak{h}_+$ and $(\lambda +\mathrm{ad}_{\kappa(r)+s})(\mathfrak{g}_+)\subseteq \mathfrak{h}_+$. The first of conditions \labelcref{RB.Lie.con.alt} implies that necessarily $r\in \mathfrak{g}_-$. In fact, up to isomorphism, one can always set $r=0$.

\begin{proposition}\label{prop.Lie.RB.lin}
Any Rota-Baxter Lie affgebra with the Rota-Baxter Lie algebra fibre $(\mathfrak{g},B)$ is isomorphic to a Rota-Baxter Lie affgebra $(\mathfrak{a}(\mathfrak{g};\kappa,\lambda,s), B)$, where necessarily
\begin{equation}\label{RBLie.red}
    \bar B\kappa B = \bar B \lambda B =0 \quad \mbox{and} \quad B(s) = s\in \im\kappa \cap \mathfrak{g}_+.
\end{equation}
\end{proposition}
\begin{proof}
    In view of the discussion at the end of Section~\cref{sec.RB} and \cite[Corollary~3.5]{AndBrzRad:Lie} any isomorphism of Rota-Baxter Lie affgebras 
    $$
    \Phi: (\mathfrak{a}(\mathfrak{g};\kappa,\lambda,s), \mathcal{B}=B+r)\longrightarrow (\mathfrak{a}(\mathfrak{g};\kappa',\lambda',s'), \mathcal{B}' = B'+r'),
    $$
    is given as $\Phi = \Psi + q$, where $q\in \mathfrak{g}$ and $\Psi$ is a Lie algebra automorphism such that
    $$
    \kappa' = \Psi\kappa\Psi^{-1}, \quad \lambda' = \Psi(\lambda - \mathrm{ad}_{q})\Psi^{-1}, \quad B'= \Psi B\Psi^{-1},
    $$
    $$
    s' = \Psi(s+q - \kappa(q)), \qquad r' = \Psi(r +\bar B(q)).
    $$
    Set $\Psi = \id$ and $q = -\kappa(r) -s$, and note that since $r$ must satisfy the first of equations \cref{RB.Lie.con}, $r=-\bar B(q)$. As a consequence
    $$
    \kappa' =\kappa,\quad \lambda' =\lambda +\mathrm{ad}_{\kappa(r)+s}, \quad B' = B,\quad s' = \kappa(s+\kappa(r) -r), \quad   r'=0.
    $$
    In view of the last of \cref{RB.Lie.con.alt}, $\bar B \lambda ' B =0$, and it is also clear that, necessarily, $s'\in \im\kappa$. Furthermore, since  $r'=0$  and $s'$ have to satisfy the first \cref{RB.Lie.con}, $B(s') = s'$. This completes the proof that, up to isomorphism, any Rota-Baxter Lie affgebra over $(\mathfrak{g},B)$  has the stated form.
\end{proof}

Note that the first two equalities in \labelcref{RBLie.red} can be stated equivalently as $\kappa(\mathfrak{g}_+) \subseteq \mathfrak{g}_+$ and $\lambda(\mathfrak{g}_+) \subseteq \mathfrak{g}_+$. \Cref{prop.Lie.RB.lin} means that, up to isomorphism, every Rota-Baxter Lie affgebra $(\mathfrak{a},\{-,-\},\mathcal{B})$ contains a point, say $e$, over which $\mathcal{B}$ is a Rota-Baxter operator over the Lie algebra $(T_e\mathfrak{a}, T_e\{-,-\})$.
In other ways, the constant part of $\mathcal{B}$ can always be gauged away. Consequently, the classification of Rota-Baxter Lie affgebras reduces to the classification of Rota-Baxter Lie algebras in which Rota-Baxter operators satisfy conditions \cref{RBLie.red} with respect to the additional data $\kappa,\lambda,s$.

\begin{example}[Rota-Baxter Lie affgebras with simple fibres]\label{ex.simple}
    Let $(\mathfrak{a},\{-,-\})$ be a complex Lie affgebra with fibres isomorphic to a simple Lie algebra $\mathfrak{g}$.  By \cite[Proposition 4.2]{AndBrzRad:Lie}, up to isomorphism the affine Lie bracket $\{-,-\}$ is necessarily of the form
    $$
    \{X,Y\} = [X,Y] + \gamma Y +s,
    $$
    for some $\gamma\in \mathbb{C}$ and $s\in \mathfrak{g}$, i.e. $\kappa = \lambda = \gamma \,\id$. In view of \Cref{prop.Lie.RB.lin}, up to isomorphism, any affine Rota-Baxter operator on $(\mathfrak{a},\{-,-\})$ is equal to a (linear) Rota-Baxter operator $B$ on $\mathfrak{g}$ such that $B(s)=s$ and $\gamma B^2 = \gamma B$. So, if $\gamma \neq 0$, $B$ is necessarily an idempotent, while for $\gamma =0$ there are no restrictions on $B$ as long as  $B(s)=s$.

    By \Cref{lem.mCYB}, up to isomorphism,  any solution to the affine modified classical Yang-Baxter equation in this case is equal to a solution $R$ of the modified classical Yang-Baxter such that $R(s)=s$ and $\gamma R^2 = \gamma \id$.

    One can explore this example further. In case of $\gamma=0$, $B$ can be taken as one of the Rota-Baxter operators derived from the Belavin-Drinfeld classification of solutions to the  modified classical Yang-Baxter equation; see \cite{BelDri:sol} for the original paper or \cite[Section~3.1]{ChaPre:gro} for a comprehensive exposition. Such operators correspond to Belavin-Drinfeld triples $(\Gamma_1,\Gamma_2,\tau)$, where $\Gamma_i$ are subsets of simple roots of $\mathfrak{g}$, and $\tau: \Gamma_1\to \Gamma_2$ is the Dynkin diagram isomorphism or bijective Killing form isometry such that for all $\alpha\in \Gamma_1$ there exists a positive integer $k$ such that $\tau^k(\alpha)\not\in \Gamma_1$. Let $\Gamma_1^+$ denote the set of positive roots generated by $\Gamma_1$. The resulting Rota-Baxter operator has the property that for all  $\alpha\not\in\Gamma_1^+$, $B(X_\alpha) = X_\alpha$, where $X_\alpha$ denotes an element of the Cartan-Weyl basis for $\mathfrak{g}$, corresponding to the root $\alpha$. In conclusion, $B$ is a Rota-Baxter operator on a Lie affgebra $\{X,Y\} = [X,Y]  +s$, provided $s\in \mathrm{span}\{X_\alpha\; |\; \mbox{$\alpha$ is positive and $\alpha\not\in \Gamma_1^+$}\}$.

    If $\gamma\neq 0$, then $B$ determines the vector space splitting  $\mathfrak{g}=\mathfrak{g}_+\oplus\mathfrak{g}_-$,  so that
    $$
    B\mid_{\mathfrak{g}_+} = \id, \quad B\mid_{\mathfrak{g}_-} = 0, \quad s\in {\mathfrak{g}_+}.
    $$
\end{example}

    We conclude by making an observation that similarly to Rota-Baxter Lie algebras, an affine Rota-Baxter operator on a Lie affgebra yields a modification of the bracket and leads to a homomorphic Rota-Baxter Lie affgebra.

\begin{proposition}\label{prop.mod.Lie}
If $(\mathfrak{a}, \{-,-\}, \mathcal{B})$ is a Rota-Baxter Lie affgebra, then the bi-affine operation, given by 
\begin{equation}\label{mod.Lie}
    \{X,Y\}_{\mathcal{B}} = \{\mathcal{B}(X),Y\} - \{X,Y\}+ \{X,\mathcal{B}(Y)\},
\end{equation}
for all $X,Y\in\mathfrak{a}$, is an affine Lie bracket on $\mathfrak{a}$. Furthermore, $\mathcal{B}$ is a homomorphism of Lie affgebras $(\mathfrak{a}, \{-,-\})\to (\mathfrak{a}, \{-,-\}_{\mathcal{B}})$, and a Rota-Baxter operator on $(\mathfrak{a}, \{-,-\}_{\mathcal{B}})$.
\end{proposition}
\begin{proof}
    This is proven by a straightforward if somewhat tedious calculation and is left to the reader as an exercise.
\end{proof}
\section*{Acknowledgments} 
The results described in this note were presented at the LMS Midlands Regional Meeting and Workshop ``Algebraic structures orbiting the Yang-Baxter equation" held in Keele, Staffordshire UK, 1--4 September 2026. I would like to thank Paul Truman for kind invitation and wonderful hospitality.

\section*{Note added} While the finishing touches were being put to the present text, the author had become aware of the arXiv preprint \cite{HanKhu:Rot} (version 1 of October 10 2026, and version 2 of October 18 2026), where Rota-Baxter operators on Lie affgebras are discussed in significant detail, and there is some overlap with Section~4 of the present note (e.g.\ conditions in \cref{RB.op.Lie} and \cref{RB.Lie.con} above are equivalent to conditions (13--14) in \cite[Corollary~2.16]{HanKhu:Rot}, while \Cref{prop.mod.Lie} is contained in \cite[Theorem~2.15]{HanKhu:Rot}).

\end{document}